\documentclass{amsart}
\usepackage[utf8]{inputenc}
\usepackage{amsmath,mathtools,amssymb,amsthm}
\usepackage{tikz-cd}
\usepackage{mathrsfs}
\usepackage[colorlinks=false, linkcolor=red]{hyperref}
\usepackage[shortlabels]{enumitem}
\usepackage{blkarray}

\theoremstyle{plain}

\newtheorem{theorem}{Theorem}[section]
\newtheorem{proposition}[theorem]{Proposition}
\newtheorem{lemma}[theorem]{Lemma}
\newtheorem{corollary}[theorem]{Corollary}
\newtheorem{remark}[theorem]{Remark}

\newcommand{\ra}{\rightarrow}

\newcommand{\OO}{\mathcal{O}}
\newcommand{\PP}{\mathbb{P}}
\newcommand{\EE}{\mathcal{E}}
\newcommand{\HH}{\mathcal{H}}

\newcommand{\PT}{\widehat{T}}

\DeclareMathOperator{\Sym}{\textrm{Sym}}
\DeclareMathOperator{\Ker}{\textrm{Ker}}
\DeclareMathOperator{\Pic}{\textrm{Pic}}

\DeclareMathOperator{\Hilb}{\textrm{Hilb}}

\title{Lines on the secant cubic hypersurfaces of Severi varieties}
\author{Renjie Lyu}
\address{MCM, Academy of Mathematics and Systems Science, Chinese Academy of Sciences, Beijing
100190, China}
\email{r.lyu@amss.ac.cn}
\date{\today}

\begin{document}

\begin{abstract} The secant varieties of Severi varieties provide special examples of (singular) cubic hypersurfaces. An interesting question asks when a given cubic hypersurface is projectively equivalent to a secant cubic hypersurface. Inspired by the ``geometric'' Torelli theorem for smooth cubic hypersurfaces due to F.~Charles, we study the geometry of lines on secant cubics, and describe the Fano variety of lines. Then we verify the ``geometric'' Torelli theorem for the case of secant cubics. Namely, a cubic hypersurface is isomorphic to a secant cubic if and only if their Fano varieties of lines are isomorphic.
\end{abstract}

\maketitle
\tableofcontents

\section*{Introduction} 
The variety of lines on a cubic hypersurface is a notable geometric object with intrinsic and extrinsic geometric features to understand the cubic. Among cubic hypersurfaces we are particularly interested in a special example, which is constructed from the Severi varieties. 

Let $k$ be a field of characteristic zero. Let $X\subset \PP^N_k$ be a nonsingular nondegenerate projective variety. We call $X^n\subset \PP^N$ a \emph{Severi variety} if $n=\frac{2(N-2)}{3}$ and the \emph{secant variety} $Sec(X)$ of $X$ 
satisfies $Sec(X)\neq \PP^N$. This notion relates to R.~Hartshorne's conjecture on complete intersections. F.~Zak solved a weak version of the conjecture by proving: 

\emph{a nonsingular nondegenerate projective variety $X^n\subset \PP^N$ with $\dim X=n>\frac{2(N-2)}{3}$ has $Sec(X)=\PP^N$,} see \cite[\S II Thm. 2.1]{Zak93}. 

Therefore the Severi variety is the boundary example among the subvarieties whose projections are closed embeddings, i.e. $Sec(X)\neq \PP^N$. For example, the Veronese surface in $\PP^5$ is a Severi variety. A classical result of F.~Severi \cite{Severi01} asserts that it is the only Severi variety in dimension $2$. T.~Fujita and J.~Robert \cite{Fuji-Rob81} found and proved the Segre fourfold $\PP^2\times \PP^2\subset \PP^8$ is the unique $4$-dimensional Severi variety. R.~Lazarsfeld \cite{LV} discovered a $16$-dimensional example: the Cartan variety $E_6\subset \PP^{26}$ given by an algebraic group of type $E_6$. From the viewpoint of representations of algebraic groups, F.~Zak provided a unified way to prove that \cite[\S IV, Thm. 4.7]{Zak93}, up to projective equivalence, there exist only four Severi varieties:
\begin{enumerate}
	\item $n=2$, $\PP^2\hookrightarrow \PP^5$, the Veronese surface;
	\item $n=4$, $\PP^2\times \PP^2\hookrightarrow \PP^8$, the Segre fourfold;
	\item $n=8$, $G(1,5)\hookrightarrow \PP^{14}$, the Pl\"ucker embedding of the Grassmannian of lines in $\PP^5$;
	\item $n=16$, $E_{6}\hookrightarrow \PP^{26}$, the closure of the orbit of a highest weight vector of the minimal representation of type $E_6$.
\end{enumerate}

The \textit{secant variety} $Sec(X)\subset \PP^N$ of the Severi variety $X$ is a singular cubic hypersurface, and the singular locus of $Sec(X)$ is exactly $X$, cf. \cite[Prop. 2.3]{ESB89} and \cite[\S IV, Thm 2.4]{Zak93}. The paper \cite{FYZ10} studies an interesting question asking when a given cubic hypersurface $V$ is projectively equivalent to the secant variety of a Severi variety. In that paper, they concern the question by investigating the invariant called ``prolongation'' on the Lie algebra of the infinitesimal automorphisms of a cubic polynomial, which is developed by J.~Hwang \cite{Hwang-prolong}. 

In this article, we intend to explore this question by the Fano variety of lines on the secant cubic hypersurface. This is inspired by the ``geometric'' Torelli theorem due to F.~Charles, which says that a smooth cubic hypersurface of dimension $d\geq 3$ is determined by the Fano variety of lines, cf. Theorem \ref{Charles-geo-Torelli} and \cite[Prop. 2.8]{Huy-cubic}. The theory of the Fano variety of lines on a smooth cubic hypersurface is systematically studied by Altman and Kielman in their work \cite{AK77}. For a smooth hypersurface $Y$ of dimension $d\geq 3$, the Fano variety $F(Y)$ of lines is a local complete intersection of the Grassmannian of lines and smooth with dimension $2d-4$. For cubics with singularities, the behavior of the variety of lines could be ``terrible'', e.g. dimension jumps and not linearly normal under the Pl\"ucker embedding. For the secant cubic hypersurface, we analyze the geometry of lines by virtue of the special geometry on the Severi variety. Though the secant cubic has relatively large singularities, its Fano variety of lines turns out not ``bad''.

The secant variety $SX$ of the Veronese surface $X\subset \PP^5$ is a cubic fourfold. A projective line $L\subset \PP^5$ contained in $SX$ is a line either secant(tangent) to $X$ or non-secant to $X$, i.e. $L\cap X=\emptyset$. By studying the non-secant lines we determine the Fano variety of lines on the secant cubic fourfold

\begin{theorem}[=Proposition \ref{lines-secant-Veronese}]\label{intro-thm1}
Let $X$ be the Veronese surface. The variety $F(SX)$ of lines on the secant variety $SX$ is the union of two irreducible components $\mathcal{C}_i, i=1,2$. The component $\mathcal{C}_1$ parametrizes all lines secant to $X$, which is isomorphic to $\Hilb^2(X)$. The family of non-secant lines in $SX$ forms a Zariski open subset of the component $\mathcal{C}_2$, and $\mathcal{C}_2$ is a $\PP^2$-bundle over $X$. Moreover, the intersection $\mathcal{C}_1\cap \mathcal{C}_2$ is the family of tangent lines to $X$.
\end{theorem}

For higher-dimensional Severi varieties, the geometry of lines in the secant variety is more subtle. For instance, unlike the case of secant cubic fourfold, there exist lines in $SX$ that transversally meet with $X$ at a single point. Through detailed analysis on the geometry of the Severi variety and the secant variety, we divide the lines on the secant variety into five different types, see Proposition \ref{5-types-line}. Based on the classification of lines of different types, we characterize the Fano variety of lines for higher dimensional secant cubics as follows. 
\begin{theorem}[cf. Theorem \ref{main-result}]\label{intro-thm2}
Let $X^n\subset \PP^N, n=\frac{2(N-2)}{3}$ be a Severi variety with $n>2$, and let $SX\subset \PP^N$ be the secant variety. Let $F(SX)$ be the variety of lines on $SX$. For $n>2$ we use $\mathcal{C}^{I},\ldots, \mathcal{C}^{V}$ to denote the strata of lines of type $I,\ldots, V$ respectively. Then we have
\begin{itemize}
	\item If $n>4$, the variety $F(SX)$ is irreducible of the dimension $2N-6$.
	\item If $n=4$, the strata of the variety $F(SX)$ satisfies
	\[
    \overline{\mathcal{C}^I}=\mathcal{C}^{I}\cup \mathcal{C}^{III}\cup \mathcal{C}^{IV}\cup\mathcal{C}^{V}, ~\overline{\mathcal{C}^{II}}=\mathcal{C}^{II}\cup \mathcal{C}^{III}\cup \mathcal{C}^{IV}\cup\mathcal{C}^{V}.
	\] 
	As a result, the variety $F(SX)$ has the dimension $10$. 
\end{itemize}
\end{theorem}

Back to our question. Theorem \ref{intro-thm1} and Theorem \ref{intro-thm2} justify the variety of lines on a secant cubic hypersurface is a local complete intersection of the Grassmannian of lines in $\PP^N$. It allows us to extend F.~Charles' conclusion to the case of this particular singular cubic:
\begin{proposition}[cf. Proposition \ref{Torelli-secant-cubic}]
Let $X^n\subset \PP^N, n=\frac{2(N-2)}{3}$ be a Severi variety, and let $SX\subset \PP^N$ be the secant cubic hypersurface. Let $Y$ be a cubic hypersurface of dimension $N-1$. Denote by $F$ and $F'$ the Fano variety of lines on $SX$ and $Y$ respectively. Then $SX$ is projectively equivalent to $Y$ if and only if $F$ is isomorphic to $F'$ with respect to the canonical Pl\"ucker polarizations.
\end{proposition}

\textbf{Acknowledgements.} The author appreciates Baohua Fu for introducing the research topics and many discussions. The author is grateful to Fyodor Zak for communications on the first draft. Many thanks to Cong Ding, Jie Liu and Ruiran Sun for some discussions.

\section{Preliminary}
Let $X$ be an irreducible nondegenerate variety in $\PP^N$, and let $Y\subset \PP^N$ be an irreducible subvariety. The join variety $S(X, Y)$ is the closure of the union of chords $\langle x, y\rangle$ for all $x\in X, y\in Y$. If $Y\subset X$, we call $S(X, Y)$ the \emph{secant variety} of $X$ and $Y$. Suppose $Y\subset X$ and $X$ is non-singular. The \emph{tangent variety} $T(X, Y)$ is the union of embedded tangent spaces $\displaystyle{\bigcup_{y\in Y}\PT_y X}$. Throughout the paper, the notation $\PT_x X$ indicates the linear subspace in $\PP^N$ that is tangent to a given subvariety $X$ at a smooth point $x\in X$. If $Y=X$, we simply denote by $SX$ (resp. $TX$) the secant (resp. tangent) variety of $X$.

Let $p\in SX\setminus X$ be a point. The \emph{secant locus} of the point $p$, denoted by $Q_p$, is the closure of the set
\begin{equation}\label{secant-loucs}
Q_p:=\{x\in X ~|~ x \text{~lies on a line secant to~} X \text{~passing through~} p \}.
\end{equation}
The \emph{contact locus} $\Sigma_p$ of a general $p\in SX\setminus X$ is defined to be 
\begin{equation}\label{contact-locus}
\Sigma_p:=\overline{\{u\in SX ~|~ \PT_u SX=\PT_p SX\}}.
\end{equation}
In other words, the tangent hyperplane $\PT_p SX$ is tangent to $SX$ along the points in $\Sigma_p$. 

Let $X^n\subset \PP^N$ be a Severi variety, i.e. nonsingular nondegenerate of dimension $n$ such that $n=\frac{2}{3}(N-2)$ and $SX\neq \PP^N$. The above two notions are fundamental to study the Severi varieties. Let us recall the following:
\begin{itemize}
\item the contact locus $\Sigma_p\subset\PP^N$ is a linear subspace of dimension $\frac{n}{2}+1$;
\item the secant locus $Q_p$ is an $\frac{n}{2}$-dimensional nonsingular quadric in $\Sigma_p$, and $Q_p=\Sigma_p\cap X$.
\end{itemize}
For reference see \cite[Thm. 2.1]{LV} or \cite[\S IV, Prop. 2.1]{Zak93}.

\begin{lemma}\label{quadric-tangent-locus} 
Let $X^n\subset \PP^N$ be a Severi variety of dimension $n$. Let $p\in SX\setminus X$, and $Q_p$ the secant locus of $p$. For any point $x\in Q_p$ there is
\[
\PT_x Q_p= \PT_x X\cap \Sigma_p.
\]
\end{lemma}

\begin{proof}
The embedded tangent space $\PT_x Q_p$ is a hyperplane of $\Sigma_p$. Assume $\PT_x Q_p$ is a proper subspace of $\PT_x X\cap \Sigma_p$. It deduces $\Sigma_p\subseteq \PT_x X$. By $\Sigma_p\cap X=Q_p$ we have $Q_p\subset \PT_x X\cap X$. Recall $\PT_x X\cap X$ is a cone with the vertex $x$. It implies $Q_p$ is a quadric cone, which leads to a contradiction since $Q_p$ is nonsingular.
\end{proof}
%Then for any $y\in Q_p$ the joining line $l:=\langle x, y\rangle$ is tangent to $X$ at $x$. In particular, the line $l$ is at least trisecant to $X$. However, $Q_p$ is a nonsingular quadric, which is a contradiction.

%The last assertion is a general fact. Consider a smooth quadric hypersurface $Q\subset \PP(V)$. The quadratic form $Q$ naturally gives an nondegenerate linear map
%\[
%\widetilde{Q}: V\ra V^*.
%\]
%Let $\omega\in V$ represents a closed point in $\PP(V)-Q$, i.e. $\widetilde{Q}(\omega, \omega)\neq 0$. Take the orthogonal complement $\omega^{\perp}:=\widetilde{Q}(\omega, \cdot)\in V^*$. Then the tangent locus 
%\[
%\Tan_{\omega}:=\{x\in Q ~|~ [\omega]\in \PT_x Q\}
%\] 
%of the point $[\omega]$ is the quadric given by $\PP(\omega^{\perp})\cap Q$. This quadric corresponds to the restriction linear map
%\[
%\widetilde{Q}|_{\omega^{\perp}}: \omega^{\perp}\ra (\omega^{\perp})^*.
%\]
%To prove that $\Tan_{\omega}$ is nonsingular, it suffices to show that $\widetilde{Q}|_{\omega^{\perp}}$ is nondegenerate. Suppose that there exists non-zero $v\in \omega^{\perp}$ such that $\widetilde{Q}(v, \omega^{\perp})=0$. Notice that $\omega\not\in \omega^{\perp}$. Hence $\omega$ and $\omega^{\perp}$ span the linear space $V$. Then the whole $V$ is annihilated by $v$ under $\widetilde{Q}$, which is impossible since $\widetilde{Q}$ is nondegenerate.

Let $X\subset \PP^N$ be a nondegenerate nonsingular variety of dimension $n$. Let $L:=\PT_x X$ be the embedded tangent space of $X$ at $x$. Consider the contact locus 
\[
Z_L:=\{y\in X ~|~ \PT_y X=L\}.
\]
Zak's theorem on tangency \cite[\S I, Thm. 1.7]{Zak93} asserts the closed subset $Z_L\subset X$ is of dimension zero. In particular, for a Severi variety $X$, we can prove that $Z_L$ consists of a single point.

\begin{lemma}\label{unique-tangency}
Let $X^n\subset \PP^N, n=\frac{2}{3}(N-2)$ be a Severi variety. Let $x, y$ be distinct points in $X$. Then $\PT_y X\neq \PT_x X$. 
\end{lemma}
\begin{proof}
Suppose there exists $y\neq x \in X$ such that $\PT_y X=\PT_x X$. For a general point $p$ in $\PT_x X-X$ we have $x, y\in Q_p$. By Lemma \ref{quadric-tangent-locus} there is $\PT_x Q_p=\PT_y Q_p$. Notice that the Gauss map 
\begin{align*}
\gamma: Q_p&\overset{\sim}{\rightarrow} Q_p^*\\
x&\ra \PT_x Q_p \nonumber
\end{align*}
for the nonsingular quadric $Q_p$ is an isomorphism. Hence $x=y$.
\end{proof}

Let $X^n\subset \PP^N$ be a Severi variety, and let $p, q\in SX\setminus X$ be two points. The definition of the contact locus \eqref{contact-locus} shows that
\[
\Sigma_p=\Sigma_q \text{~if and only if~} \PT_p SX=\PT_q SX.
\] 
In this situation, we have $Q_p=Q_q$. When $Q_p\neq Q_q$, the intersection of $Q_p$ and $Q_q$ is special. The following result frequently occurs in the rest of the paper.

\begin{theorem}\cite[\S IV, Cor. 3.5]{Zak93}\label{intersect-secant-quadric}
Let $X^n\subset \PP^N$ be a Severi variety, and let $p, q\in SX\setminus X$ be two points. Suppose that $Q_p\neq Q_q$. Then the intersection $Q_p\cap Q_q$ is either
\begin{enumerate}
	\item a single point, for general $p, q\in SX\setminus X$, 
	\item or a linear subspace of dimension $\frac{n}{4}$.
\end{enumerate}
\end{theorem}

The following lemma presents more refined information on the intersection of secant loci, which is useful in Section $3$. The set-up is as follows.

Let $X$ be a Severi variety, and let $x\in X$ be an arbitrary point. We set 
\begin{equation}\label{family-secant-loci}
Q^x:=\{[\PT_z SX]\in {\PP^N}^* ~|~ z\in S(x,X)\setminus X\}.
\end{equation}
The variety $Q^x\subset {\PP^N}^*$ is the image of $S(x, X)$ under the $(N-1)$-Gauss map on $SX$. Moreover, it is a nonsingular $\frac{n}{2}$-dimensional quadric in the dual projective space $(\PT_x X)^*$, see \cite[\S IV, Prop. 3.1]{Zak93}. Under the correspondence $[\PT_z SX]\leadsto Q_z$, the variety $Q^x$ can be regarded as the family of $\frac{n}{2}$-dimensional quadrics in $X$ passing through the point $x$. The following statement characterizes the intersection of two secant loci passing through the point $x$.
\begin{lemma}\cite[\S IV, Prop. 3.3]{Zak93}\label{family-secant-loci-cor}
Fix $p\in S(x, X)\setminus X$ and the corresponding point $[\PT_p SX]\in Q^x$. For a point $q\in S(x, X)\setminus X$ and the corresponding $[\PT_q SX]\in Q^x$, the dimension of the linear subspace $Q_p\cap Q_q$ is positive if and only if $[\PT_q SX]$ lies in the tangent cone of $Q^x$ with the vertex $[\PT_p SX]$.
\end{lemma}

\section{Lines on secant cubic fourfolds}\label{Veronese-surface}
In this section, we focus on the Veronese surface, the two-dimensional Severi variety. The goal is to describe the variety of lines on the secant variety of the Veronese surface. We first recall some basic settings for Severi varieties.

Let $X^n\subset \PP^N$ be a Severi variety, and denote by $SX\subset \PP^N$ the secant variety. Consider the $(N-1)$-th Gauss map 
\begin{equation}\label{Gauss-map}
\gamma_{N-1}: SX\dashrightarrow SX^*,~ \gamma_{N-1}(z)=[\PT_z SX]
\end{equation}
of the hypersurface $SX$. Obviously the map $\gamma_{N-1}$ is not defined on the singular locus $X\subset SX$. Here $SX^*$ denotes the dual variety of $SX$, which is isomorphic to $X$, see \cite[\S III, Thm. 2.4]{Zak93}. The conormal variety
\[
\mathcal{C}:=\overline{\{(p, [H])\in (SX\setminus X)\times {\PP^N}^*~|~ \PT_p SX=H\}}
\]
of $SX$ gives a resolution of $\gamma_{N-1}$. The first projection $\mathcal{C}\ra SX$ can be viewed as the blowing-up of $SX$ along $X$. For the second projection $p: \mathcal{C}\ra SX^*$, the fibre of $p$ over a point $[\PT_z SX]\in (SX)^*$ is exactly the contact locus $\Sigma_z$ described in \eqref{contact-locus}. Let $\mathcal{E}$ be the exceptional divisor of the blowing-up $\mathcal{C}\ra SX$. The restriction of the fibre $\Sigma_z$ to $\mathcal{E}$ is the secant locus $Q_z$ as in \eqref{secant-loucs}, i.e.
\[
\mathcal{E}\cap \Sigma_z=X\cap \Sigma_z=Q_z,
\]
Hence $\mathcal{E}\ra SX^*$ is a smooth family of secant quadrics.

In order to study lines on the secant variety of the Veronese surface, we use a concrete description of the Gauss map $\gamma_{N-1}$ as follows.

Let $W$ be a $3$-dimensional vector space over an algebraically closed field $k$ with char$(k)=0$. Let us regard $\PP^5$ as the projective space $\PP(\Sym^2(W^*))$ of symmetric bilinear forms(or symmetric $3\times 3$ matrices) on $W$. Then the Veronese surface $X\subset \PP^5(\Sym^2 W^*)$ corresponds to the cone of the matrices of rank less than or equal to one, and the secant variety $SX\subset \PP(\Sym^2 W^*)$ corresponds to the cone of degenerate symmetric matrices.

Assume $[\varphi]\in SX\setminus X$ is a symmetric bilinear form on $W$. The kernel of $\varphi$ is a one-dimensional subspace of $W$. Therefore it defines a rational map
\begin{equation}\label{exp-Gauss-map}
\gamma: SX\dashrightarrow \PP(W), \varphi\mapsto [\Ker(\varphi)].
\end{equation}
The map $\gamma$ is not defined along the Veronese surface $X$ since the corresponding bilinear forms have rank one. The graph $\mathcal{H}$ of the rational map $\gamma$
\[
\begin{tikzcd}
& \mathcal{H}:=\{([\varphi], [v])\in SX\times \PP(W)~|~ \varphi(v)=0\} \ar[ld, "\tau"'] \ar[rd, "p"]&\\
SX && \PP(W)
\end{tikzcd}
\]
is a resolution of $\gamma$. The morphism $\tau: \HH\ra SX$ is the blowing-up of $SX$ along the singular locus $X$. The exceptional divisor $\EE$ of the blow-up can be realized as the incidence correspondence
\[
\EE\cong \{([\phi], [v])\in \PP(W^*)\times \PP(W) ~|~ \phi(v)=0\}.
\]
Therefore we can view the graph $\HH$ as the conormal variety $\mathcal{C}$ since both are blow-ups of $SX$ along $X$. 

%There is more direct way to identify $\gamma$ with the Gauss map. The equation of $SX$ is the determinant of the symmetric matrix $(x_{ij})_{1\leq i\leq j\leq 3}$. For a symmetric matrix $\varphi:=(a_{ij})_{1\leq i\leq j\leq 3}$, the tangent hyperplane of $SX$ at $\varphi$ is given by the adjugate matrix of $\varphi$. It is easy to verify that $\gamma$ sends the symmetric matrix $(a_{ij})_{1\leq i\leq j\leq 3}$ to its adjugate matrix via the standard Veronese embedding $\PP(W)\hookrightarrow \PP(\Sym^2 W)$.

Let $v\in W$ be any non-zero vector. Let
\[
\text{Ann}(v):=\{\varphi\in SX~|~ \varphi(v)=0\}
\]
be the cone of the symmetric bilinear forms annihilated by $v$. Suppose that $[v]$ represents a tangent hyperplane $[\PT_z SX]\in SX^*\cong \PP(W)$. By the identification $\mathcal{C}\cong \HH$ we have 
\begin{equation}\label{contact-Ann}
\Sigma_z= \PP(\text{Ann}(v)).
\end{equation} 

\begin{lemma}
Let $L\subset SX\setminus X$ be any non-secant line in $SX$. Then the image $\gamma(L)$ is a projective line in $\PP(W)$, i.e. $\deg_H \gamma(L)=1$ where $H$ is the hyperplane section class of $\PP(W)$.
\end{lemma}
\begin{proof}
We denote by $C$ the curve $\gamma(L)$ in $\PP(W)$. Let $H_2$ be the hyperplane section class on $\PP(\Sym^2 W)$. The Veronese embedding $\PP(W)\hookrightarrow \PP(\Sym^2 W)$ is given by the linear system of quadrics on $\PP(W)$. Hence it suffices to show 
\[
\deg C\cdot H_2=2.
\]

Let $\widetilde{L}$ be the proper transform of $L$ under the blowing-up $\tau$. By the projection formula we have
\[
C\cdot H_2= \widetilde{L}\cdot p^*H_2=L\cdot \tau_*p^*H_2.
\]
The second equality holds since $L\cap X=\emptyset$ and thus $\tau^*L=\widetilde{L}$. Let $H_1$ be the hyperplane section class on $\PP(\Sym^2 W^*)$. From the work of Ein and Shepherd-Barron, see \cite[(2.0.3)]{ESB89}, one can deduce the following relation
\[
\EE=2\tau^*H_1-p^*H_2 \text{~in~} \Pic(\HH).
\]
Here the notation $H_1$ (resp. $H_2$) indicates the induced hyperplane class on $SX$ (resp. $SX^*$). It follows that
\[
C\cdot H_2=L\cdot (2H_1-\tau_*\EE)=2.
\]
\end{proof}

\begin{remark}
In fact, there is a more direct way to show $\gamma(L)$ is a projective line in $\PP(W)$. Let $\varphi_1, \varphi_2$ be two symmetric bilinear forms that correspond to any two points in $L$. By the deifnition of the map \eqref{exp-Gauss-map}, the curve $\gamma(L)$ is represented by the hyperplane
\[
\langle \Ker(\varphi_1), \Ker(\varphi_2) \rangle.
\] 
of $W$ generated by the kernels of $\varphi_1, \varphi_2$. It is easy to check the hyperplane is independent of the choice of $\varphi_1, \varphi_2$ in $L$.
\end{remark}

The above equivalent characterization of non-secant lines arises the following key observation. It is essential to determine the geometric shape of the family of non-secant lines. 
\begin{lemma}\label{line-tangent-space}
Let $X$ be the Veronese surface, and let $SX$ be the secant variety. For any non-secant line $L\subset SX\setminus X$, there exists a unique $x\in X$ such that $L\subset \PT_x X$.
\end{lemma}
\begin{proof}
Let us set $U\subset W$ to be the hyperplane associated to the projective line $\gamma(L)\subset \PP(W)$. Let $\phi\in W^*$ be the linear map annihilated on $U$. Denote by $x\in X\cong \PP(W^*)$ the point corresponding to $\phi$. We claim that $L$ is contained in the embedded tangent space of $X$ at $x$. 

For any $p\in L$, as shown in \eqref{contact-Ann} we have $\Sigma_p=\PP(\text{Ann}(\gamma(p)))$. It follows that
\begin{equation}\label{single-intersection-secant}
\bigcap_{p\in L}\Sigma_p=\PP(\text{Ann}(U))=[\phi].
\end{equation}
The line $L$ is not secant to $X$. Then for any $p\neq q\in L$ we have $\Sigma_p\neq \Sigma_q$. It implies $\Sigma_p\cap \Sigma_q\subset X$, and thus
\begin{equation}\label{equal-intersection}
\Sigma_p\cap \Sigma_q=\Sigma_p\cap \Sigma_q\cap X=Q_p\cap Q_q.
\end{equation}
By Theorem \ref{intersect-secant-quadric} the intersection $Q_p\cap Q_q$ is either a single point or a linear subspace. In two-dimensional case, the secant locus is a smooth conic. Therefore $Q_p\cap Q_q$ consists of a single point. Combine \eqref{single-intersection-secant} and \eqref{equal-intersection} the point must be $\{x\}$ for any $p, q\in L$. 

In the next we show that all points $p\in L$ lie in $\PT_x X$. Suppose the lines $\langle x, p\rangle$ and $\langle x, q\rangle$ of two distinct points $p, q\in L$ are not tangent to $X$ at $x$. Then they are properly secant to $X$ so that $\langle x, p\rangle$ and $\langle x, q\rangle$ meet with $X$ at another points $y$ and $z$ respectively. We can see that the line $\langle y, z\rangle$ must intersect with $L$ at another point $r\neq p$. Then we obtain $y\in Q_r\cap Q_p$, which implies $y=x$. Therefore the chord $\langle x, p\rangle$ must be tangent to $X$ at $x$, which implies $L\subset \PT_x X$.

The point $x$ is the unique point of $X$ satisfying $L\subset \PT_x X$. Indeed for any $y\in X$ satisfying $L\subset \PT_y X$, we have $y\in Q_p\cap Q_q$ for all $p, q\in L$.
\end{proof}

\begin{remark}\label{unique-tangent-intersect}
For any $x\in X$, the embedded tangent space $\PT_x X$ meets with $X$ only at $x$. Otherwise assume there exists another $y\in \PT_x X\cap X$. Since $X$ contains no straight line in $\PP^5$, we can pick a point $p\in \langle x, y \rangle-X$. Then the chord $\langle x, y \rangle$ is tangent to the smooth conic $Q_p$ at $x$, which could happen if and only if $x=y$.
\end{remark}

\begin{proposition}\label{lines-secant-Veronese}
Let $X$ be the Veronese surface. The variety $F(SX)$ of lines on the secant variety $SX$ is the union of two irreducible components $\mathcal{C}_i, i=1,2$. The component $\mathcal{C}_1$ parametrizes all lines secant to $X$, which is isomorphic to $\Hilb^2(X)$. The family of non-secant lines in $SX$ forms a Zariski open subset of the component $\mathcal{C}_2$, and $\mathcal{C}_2$ is a $\PP^2$-bundle over $X$. Moreover, the intersection $\mathcal{C}_1\cap \mathcal{C}_2$ is the family of tangent lines to $X$.
\end{proposition}
\begin{proof}
Any line secant to $X$ is determined by two distinct points or a direction supported at one point in $X$. Hence the component $\mathcal{C}_1$ of the secant lines is isomorphic to $\text{Hilb}^2(X)$. 

For the component of non-secant lines, we consider the embedded tangent bundle
\[
\PT X:=\{(x, p)\in X\times \PP^5 ~|~ p\in \PT_x X\}
\] 
of $X$. The image of the projection $\psi: \PT X\ra \PP^5$ is the tangential variety $TX$ of $X$. Let $Gr(\PT X/X)$ be the Grassmannian of lines relatively to the tangent bundle $\PT X\ra X$. It parametrizes the lines in $SX$ lying in embedded tangent spaces of $X$. For a Severi variety $X$ we have $SX=TX$. Hence the projection $\psi$ induces a morphism
\[
\Psi: Gr(\PT X/X)\ra F(SX), (x, [L\subset\PT_x X])\mapsto [L].
\]
By the result of Lemma \ref{line-tangent-space}, the morphism $\Psi$ is injective. We define the image of $\Psi$ to be the component $\mathcal{C}_2$ in our assertion. Since the embedded tangent space $\PT_x X$ is isomorphic to $\PP^2$, the Grassmannian $Gr(\PT X/X)$ of lines is isomorphic to a $\PP^2$-bundle over $X$. 

By Remark \ref{unique-tangent-intersect}, any line $[L]\in \mathcal{C}_2$ meeting with $X$ must be a tangent line of $X$ passing through some $x$. Therefore $\mathcal{C}_1\cap \mathcal{C}_2$ is the family of tangent lines of $X$, which is isomorphic to a $\PP^1$-bundle over $X$.
\end{proof}

\section{Geometry of lines on the secant variety} \label{projection}
As we have seen, the key observation to study the lines on the secant variety of the Veronese surface is:

$(\star)$ If $L\subset SX\setminus X$ is a non-secant line, there exists a unique $x\in X$ lying in the secant locus $Q_p$ for all $p\in L$, and $L$ is contained in the embedded tangent space $\PT_x X$.
%This property is useful to describe the variety of lines on the secant variety. The observation provides a guideline for our study in this section. 

This observation seems useful to describe the variety of lines on the secant variety. Unfortunately, it is not always true for higher dimensional cases. Later we will find for the Serge fourfold $X$, there exist non-secant lines $L\subset SX\setminus X$ not contained in any embedded tangent space, see Proposition \ref{lines-in-tangent}. Nevertheless, this observation provides a guideline for our study in higher dimensional cases. In Proposition \ref{5-types-line} we divide the lines on the secant varieties into five types, unlike the case of the Veronese surface, which only has two types---secant and non-secant. Then it is not hard to check for which type of lines are contained in embedded tangent spaces, see Proposition \ref{lines-in-tangent}. At the end, we describe the variety of lines on the secant varieties, see Theorem \ref{main-result}.  
%Nevertheless, for Severi varieties of dimension greater than $4$, we verify the property for generic lines on the secant variety, see Proposition \ref{lines-in-tangent}. For instance, there exists one type of line $L\subset SX$ that transversally intersects with $X$ at one point. That is to say, the line $L$ is not secant to $X$ but $L\cap X\neq \emptyset$, see Proposition \ref{5-types-line}.

Suppose $L\subset SX\setminus X$ is any non-secant line. Let $p\in L$ be any point. We have $L\subset \PT_p SX$ and $\Sigma_p\cap L=\{p\}$. Consider the projection map
\begin{equation}\label{secant-proj}
\pi_p: SX\cap \PT_p SX\dashrightarrow \PP^{n-1}
\end{equation}
along the contact locus $\Sigma_p$. Then $L$ lies in one fibre of the projection. In such a way, all non-secant lines in $SX$ is contained in fibres of such projections. Analyzing the positions of lines in fibres leads to the desired classification of lines in Proposition \ref{5-types-line}.

%For simplicity we denote by $H_p$ the tangent hyperplane $\PT_p SX$. Let $Bl_{\Sigma_p} SX\cap H_p$ be the blow-up of $SX\cap H_p$ along $\Sigma_p$. By the definition \eqref{contact-locus} of the contact locus, $H_p$ is tangent to $SX$ along $\Sigma_p$. Since $SX\cap H_p$ is a hypersurface of degree $3$ the induced morphism 
%\[
%\tilde{\pi}_p: Bl_{\Sigma_p} SX\cap H_p\ra \PP^{n-1}
%\]

\begin{lemma}\label{fibre-projection}
Let $X^n\subset \PP^N, n=\frac{2}{3}(N-2)$ be a Severi variety of dimension $n>2$. Fix a point $p\in SX\setminus X$. Denote by $H_p\subset \PP^N$ the embedded tangent hyperplane $\PT_p SX$ at $p$, and $\Sigma_p$ the contact locus of $H_p$. Let $Bl_{\Sigma_p} SX\cap H_p$ be the blow-up of $SX\cap H_p$ along $\Sigma_p$. Consider the induced morphism 
\[
\tilde{\pi}_p: Bl_{\Sigma_p} SX\cap H_p\ra \PP^{n-1}
\]
For any $u\in (SX\cap H_p)\setminus \Sigma_p$, the fibre of $\tilde{\pi}_p$ over the point $\pi_p(u)$ is either
\begin{enumerate}[1)]
\item $S(u, \PT_x Q_p)$ where $u\in \PT_x X$ and such $x$ is unique;
\item or the linear space $S(u, \Sigma_p)$.
\end{enumerate}
\end{lemma}
\begin{proof}
The fibre $\tilde{\pi}_p^{-1}(\pi_p(u))$ is obtained from the intersection 
\[
S(u, \Sigma_p)\cap SX.
\]
It is a cubic hypersurface if $SX$ properly intersects with $S(u, \Sigma_p)$. Otherwise $S(u, \Sigma_p)\subset SX$. 

Let us consider the first case. Notice that $H_p$ is tangent to $SX$ along $\Sigma_p$. Then the cubic $SX\cap S(u, \Sigma_p)$ is singular along $\Sigma_p$. It follows that the intersection$SX\cap S(u, \Sigma_p)$ is the union of $\Sigma_p$ with multiplicity $2$ and a residue hyperplane $P$. We show that $P=S(u, \PT_x Q_p)$ as follows.

According to \cite[\S IV, Rem. 4.3]{Zak93}, the hyperplane section $SX\cap H_p$ can be written as
\[
SX\cap H_p=T(Q_p, X)=\underset{x\in Q_p}{\bigcup} \PT_x X.
\] 
Recall Lemma \ref{quadric-tangent-locus}
\[
\Sigma_p\cap \PT_x X=\PT_x Q_p.
\] 
Then the restriction of the projection \eqref{secant-proj} on each embedded tangent space $\PT_x X$ can be viewed as
\[
\PT_x X\dashrightarrow \PT_x X/ \PT_x Q_p
\]
For any point $u\in SX\cap H_p\setminus \Sigma_p$, there exists some $x\in Q_p$ such that $u\in \PT_x X$. Then the linear subspace $S(u, \PT_x Q_p)$ is contained in $S(u, \Sigma_p)$, which deduces
\[
SX\cap S(u, \Sigma_p)=2\Sigma_p \cup S(u, \PT_x Q_p)
\]
and $P=S(u, \PT_x Q_p)$. Now it is not hard to show that such $x\in Q_p$ is unique. Assume that there exists another point $y\in Q_p$ such that $u\in \PT_y X$. Then we have
\[
SX\cap S(u, \Sigma_p)=2\Sigma_p \cup S(u, \PT_y Q_p).
\] 
It implies $S(u, \PT_x Q_p)=S(u, \PT_y Q_p)$, and in particular $\PT_x Q_p=\PT_y Q_p$. Then we obtain $x=y$ by Lemma \ref{unique-tangency}.

The case $2)$ may exists. Assume that the linear subspace $S(u, \Sigma_p)$ contains a point $x\in X\cap H_p-Q_p$. The cubic $SX\cap H_p$ is singular along $\Sigma_p$ and $x$. Therefore any line passing through $x$ and $\Sigma_p$ is contained in $SX\cap H_p$. It implies 
\[
S(x, \Sigma_p)=S(u, \Sigma_p)\subset SX\cap H_p.
\] 
\end{proof}

At the end of the above proof, the condition $S(u, \Sigma_p)\cap X-Q_p\neq \emptyset$ is sufficient for $S(u, \Sigma_p)\subset SX\cap H_p$. We show it is also a necessary condition in the next lemma.

\begin{lemma}\label{second-fibre}
Keep the same notations in Lemma \ref{fibre-projection}. Let $X$ be a Severi variety with $\dim X>2$ For any $u\in (SX\cap H_p)\setminus \Sigma_p$, the linear subspace $S(u, \Sigma_p)$ is contained in $SX$ if and only if 
\[
\exists z\in S(u, \Sigma_p)\cap X-Q_p.
\]
\end{lemma}
The proof of the lemma involves a few settings and results based on Zak's work \cite{Zak93}. We postpone the proof and provide necessary ingredients in the following.

Fix a point $x\in X$. Recall from \eqref{family-secant-loci} the variety $Q^x$ denotes the family of $\frac{n}{2}$-dimensional quadrics passing through $x$. Let $p,q\in S(x, X)\setminus X$ such that $Q_p\cap Q_q=\{x\}$, and denote by $C_p$ (resp. $C_q$) the tangent cone $\PT_x Q_p\cap Q_p$ (resp. $\PT_x Q_q\cap Q_q$). The join variety $S(C_p, C_q)$ is not contained in $X$ for $n>2$. For $z\in S(C_p, C_q)-X$, the linear subspaces $Q_p\cap Q_z$ and $Q_q\cap Q_z$ have positive dimensions. It therefore defines a morphism 
\begin{equation}\label{join-cone}
S(C_p, C_q)-X\ra Q^{p, q}
\end{equation}
where $Q^{p, q}\subset Q^x$ denotes the subfamily parametrizing the quadrics that passing through $x$ and meeting with $Q_p$ and $Q_q$ along positive dimensional subspaces. The morphism \eqref{join-cone} is surjective. In fact, for any quadric $[Q]\in Q^x$ with $\dim Q\cap Q_p>0$ and $\dim Q\cap Q_q>0$, there exists a point $z\in S(C_p, C_q)-X$ such that $Q_z=Q$, cf. \cite[Prop. \S IV, 3.3]{Zak93}
 
Suppose that $p, u\in SX\setminus X$ such that $\dim Q_u\cap Q_p>0$. We hope to find an $x\in Q_u\cap Q_p$ and $[Q_q]\in Q^x$ such that $u\in S(C_p, C_q)-X$.
\begin{lemma}\label{positive-intersect}
Let $p, u\in SX\setminus X$. Suppose that the linear subspace $Q_p\cap Q_u$ has positive dimension. Then there exists $x\in Q_p\cap Q_u$ and $[Q_q]\in Q^x$ such that $u\in S(C_p, C_q)-X$.
\end{lemma}
\begin{proof}
For any $x\in Q_p\cap Q_u$, the space $Q^x$ is a smooth quadric with $\dim Q^x\geq 2$. Then there exists $[\PT_q SX]\in Q^x$ lying in the tangent cone of $Q^x$ at the vertex $[\PT_u SX]$ but not in the tangent cone of $Q^x$ at the vertex $[\PT_p SX]$. Lemma \ref{family-secant-loci-cor} asserts that the secant locus $Q_q$ satisfies 
\[
Q_q\cap Q_p=\{x\}, ~\dim Q_q\cap Q_u>0.
\] 

Consider the morphism $S(C_p, C_q)-X\ra Q^{p, q}$. Let us set 
\[
\PP^{\alpha_1}=Q_p\cap Q_u, ~\PP^{\alpha_2}=Q_q\cap Q_u, ~\alpha_1=\alpha_2=\frac{n}{4}.
\]
The set of points $z\in S(C_p, C_q)-X$ whose secant locus is $Q_u$ is the linear subspace $S(\PP^{\alpha_1}, \PP^{\alpha_2})$. Notice that 
\[
\PP^{\alpha_1}\cap \PP^{\alpha_2}=\{x\},~\text{and ~}
\PP^{\alpha_1}, \PP^{\alpha_2}\subset \PT_x Q_u.
\] 
By dimension counting we have 
\[
S(\PP^{\alpha_1}, \PP^{\alpha_2})=\PT_x Q_u\subset \Sigma_u
\]
as a hyperplane in $\Sigma_p$.

The tangent hyperplane $\PT_x Q_u$ may not contain the point $u$. However, we could vary the point $x$ in $Q_p\cap Q_u$ to obtain a family of tangent hyperplanes $\PT_x Q_u$. Then the union of $\PT_x Q_u$ will cover $\Sigma_u$. Therefore, there must exists some $x\in Q_p\cap Q_u$ such that $u\in \PT_x Q_u$. For such $x$ and suitable $[Q_q]\in Q^x$ we have $u\in S(C_p, C_q)$. 
\end{proof}

%\begin{lemma}\label{jump-dim}
%Let $u\in SX\cap H_p\setminus \Sigma_p$. We denote by $M_u$ the $(\frac{n}{2}+2)$-dimensional projective space $S(u, \Sigma_p)$. For $x\in Q_p$, we have 
%\[
%M_u\cap \PT_x X=
%\begin{cases}
%\PT_x Q_p, & \textrm{if~} x \textrm{~is general,~}\\
%S(\tilde{u}, \PT_x Q_p) \textrm{~for some~} \tilde{u}\in \PT_x X\setminus \PT_x Q_p, & otherwise. 
%\end{cases}
%\]
%\end{lemma}
%\begin{proof}
%Recall that $\Sigma_p\cap \PT_x X=\PT_x Q_p$ for any $x\in Q_p$. Hence we have 
%\[
%\PT_x Q_p\leq M_u\cap \PT_x X.
%\]
%By the reason of dimension, the linear subspace $M_u\cap \PT_x X$ is either equal to $\PT_x Q_p$ or a hyperplane of $M_u$ containing $\PT_x Q_p$. It deduce from Lemma \ref{intesect-tangent-plane} that for a general pair $(x, y)$ of points in $Q_p$, we have
%\[
%(M_u\cap \PT_x X)\cap (M_u\cap \PT_y X)=\PT_x Q_p\cap \PT_y Q_p.
%\]
%Hence for a general $x\in Q_p$, the linear subspace $M_u\cap \PT_x X$ cannot be a hyperplane of $M_u$ since $\text{codim}(\PT_x Q_p\cap \PT_y Q_p, M_u)=2$. When $M_u\cap \PT_x X$ is a hyperplane, it is the join variety of $\PT_x Q_p$ and some point $\tilde{u}\in\PT_x X\setminus \PT_x Q_p$
%\end{proof}

Now we can finish the proof of Lemma \ref{second-fibre}.

\begin{proof}[Proof of Lemma \ref{second-fibre}] 
The upshot is to prove
\[
\exists ~\tilde{u}\in S(u, \Sigma_p)-\Sigma_p \text{~such that~} \dim Q_p\cap Q_{\tilde{u}}>0.
\]
Then Lemma \ref{positive-intersect} concludes the existence of a point $x\in Q_{\tilde{u}}\cap Q_p$ and $q\in S(x, X)\setminus X$ such that 
\[
Q_q\cap Q_p=\{x\}, ~\tilde{u}\in S(C_p, C_q)-X.
\] 
where $C_p$ (resp. $C_q$) is the tangent cone of $Q_p$ (resp. $Q_q$) at $x$. Then any point $z\in Q_{\tilde{u}}\cap Q_q$ satisfies $z\in S(\tilde{u}, \Sigma_p)\cap X-Q_p=S(u, \Sigma_p)\cap X-Q_p$.

To simplify the notation, we denote by $M_u$ the linear subspace $S(u, \Sigma_p)$. Recall that $\Sigma_p\cap \PT_x X=\PT_x Q_p$ for any $x\in Q_p$. Hence
\[
\PT_x Q_p\leq M_u\cap \PT_x X.
\]
By the reason of dimension, the linear space $M_u\cap \PT_x X$ is either equal to $\PT_x Q_p$ or a hyperplane $S(\tilde{u},\PT_x Q_p)$ of $M_u$ for some $\tilde{u}\in \PT_x X\setminus \PT_x Q_p$. Under the hypothesis $M_u\subset SX$, we have
\[
M_u=M_u\cap SX=M_u\cap T(Q_p, X)=\bigcup_{x\in Q_p} M_u\cap \PT_x X.
\] 
There must exist more than one $x\in Q_p$ such that $M_u\cap \PT_x X$ is a hyperplane of $M_u$. Otherwise the union $\bigcup_{x\in Q_p} M_u\cap \PT_x X$ is equal to $\Sigma_p\cup (M_u\cap \PT_x X)$ for a single $x\in Q_p$, which cannot be $M_u$.

Let $x, y\in Q_p$ be distinct points such that $M_u\cap \PT_x X$ and $M_u\cap \PT_y X$ are hyperplanes. Let us write
\[
M_u\cap \PT_x X=S(u_x, \PT_x Q_p), ~M_u\cap \PT_y X=S(u_y, \PT_y Q_p)
\]
for some $u_x\in \PT_x X\setminus \PT_x Q_p, u_y\in \PT_y X\setminus \PT_y Q_p$. The subspace  
\[
S(u_x, \PT_x Q_p)\cap S(u_y, \PT_y Q_p)
\] 
of $M_u$ is not contained in $\Sigma_p$. Let $\tilde{u}\in M_u-\Sigma_p$ be any common point of $S(u_x, \PT_x Q_p)$ and $S(u_y, \PT_y Q_p)$. If $\tilde{u}\in X$, we are done. Otherwise suppose that $\tilde{u}\in SX\setminus X$. Then $x, y$ both lie in the secant locus $Q_{\tilde{u}}$ of $\tilde{u}$. By Theorem \ref{intersect-secant-quadric} the linear space $Q_{\tilde{u}}\cap Q_p$ has positive dimension.
\end{proof} 

Theorem \ref{intersect-secant-quadric} shows that the intersection $Q_p\cap Q_q$ is either a single point or a linear space with positive dimension. A further question may arise: 

\emph{Fix any $p\in SX\setminus X$, what $q\in SX\setminus X$ satisfies $\dim Q_p\cap Q_q>0$.} 

In \cite[Cor. 2.8]{LV}, Lazarsfeld and Van de Ven proved that $Q_p\cap Q_q$ consists of a single point if $q\notin \PT_p SX$. Hence $q\in \PT_p SX$ is a necessary condition for the above question. In the following corollary, we close the question based on the results of Lemma \ref{fibre-projection} and \ref{second-fibre}. Later the corollary is crucial to characterize lines of different types on secant varieties.

\begin{corollary}\label{position-positive-intersection}
Suppose that $q\in (SX\cap H_p)\setminus \Sigma_p$. If $q$ lies in the fibre of type $1)$ in Lemma \ref{fibre-projection}, then $Q_p\cap Q_q$ is a single point. If $q$ lies in the fibre of type $2)$ in Lemma \ref{fibre-projection}, then $Q_p\cap Q_q$ has positive dimension.
\end{corollary}
%Let $u\in (SX\cap H_p)\setminus \Sigma_p$. If the fibre over $\pi_p(u)$ is the case $(1)$ in the conclusion of Lemma \ref{fibre-projection}, then for any $q\in S(u, \PT_x Q_p)-\Sigma_p$, the intersection $Q_p\cap Q_q$ consists of the single point $\{x\}$. If the fibre over $\pi_p(u)$ is the case $(2)$, then for any $q\in S(u, \Sigma_p)\setminus(\Sigma_p\cup X)$, the linear subspace $Q_p\cap Q_q$ has positive dimension.

\begin{proof}
We prove the first assertion by contradiction. Assume the dimension of $Q_q\cap Q_p$ is positive. Lemma \ref{second-fibre} asserts the space $\in S(q, \Sigma_p)$ meets with $X$ away from $Q_p$, and $S(q, \Sigma_p)\subset SX$. It contradicts our assumption. Therefore, $Q_p\cap Q_q$ is a single point. 

For the second assertion, consider the diagram
\[
\begin{tikzcd}
X\cap H_p \ar[r, hook] \ar[rd, dashrightarrow, "\pi"'] & SX\cap H_p \ar[d, dashrightarrow, "\pi_p"]\\
& \PP^{n-1}.
\end{tikzcd}
\]
We have shown that $S(q, \Sigma_p)$ contains some point $z\in X\setminus Q_p$. The closure of the fibre $\pi^{-1}(\pi(z))$ is an $(\frac{n}{4}+1)$-dimensional linear subspace $S(z, \Lambda)$ which intersects $Q_p$ along a maximal linear subspace $\Lambda\cong \PP^{\frac{n}{4}}$, see \cite[\S IV, Lem. 4.1]{Zak93}. The diagram shows
\[
S(q, \Sigma_p)\cap X=Q_p\cup S(z, \Lambda).
\] 
Consider the quadric cone $C_q Q_p\subset S(q, \Sigma_p)$ over $Q_p$ with the vertex $q$. We have 
\[
C_q Q_p\cap S(z, \Lambda)=\Lambda\cup \Lambda'
\]
where $\Lambda'$ is the residue hyperplane. The projection along $q$ in the linear space $S(q, \Sigma_p)$ maps $\Lambda'$ to an $\frac{n}{4}$-dimensional linear subspace $\tilde{\Lambda}\subset \Sigma_p$. It is easy to see that
\[
\Lambda'\cup \tilde{\Lambda}\subset Q_q, \textrm{~and~} \tilde{\Lambda}=Q_p\cap Q_q.
\]
Therefore our assertion follows.
\end{proof}

Lemma \ref{fibre-projection} and Lemma \ref{second-fibre} help us to explicitly characterize the lines in $SX$ of different types.  
\begin{proposition}[Definition]\label{5-types-line}
Let $X$ be a Severi variety with $n=\dim X>2$. The lines on the secant variety $SX$ can be divided into the following five types. 
\begin{enumerate}[(I)]
\item Non-secant line $L\subset SX\setminus X$. For any point $p\in L$ with the projection \eqref{secant-proj}, the line $L$ is contained in the fibre of type $1)$ in Lemma \ref{fibre-projection}. 
\item Non-secant line $L\subset SX\setminus X$. For any point $p\in L$ with the projection \eqref{secant-proj}, the line $L$ is contained in the fibre of type $2)$ in Lemma \ref{fibre-projection}. 
\item Line $L$ transversally meets with $X$ at one point. For any point $p\in L-X$ with the projection \eqref{secant-proj}, the line $L$ is contained in the fibre of type $2)$ in Lemma \ref{fibre-projection}, and the length of $L\cap X$ is equal to one. 
\item The secant lines of $X$. 
\item Lines contained in $X$.
\end{enumerate}
\end{proposition}

\begin{proof}
We need to check that the descriptions of lines of type $I, II$, and $III$ are independent of projections \eqref{secant-proj}. 

Suppose that the line $L$ is of type $I$. If we replace the point $p$ by any other $q\in L$, and consider the projection $\pi_q$ along the contact locus $\Sigma_q$. It follows from Corollary \ref{position-positive-intersection} that $Q_q\cap Q_p$ consists of a single point, and the fibre of $\pi_q$ containing $L$ is of the first form.

The similar argument applies to the line $L$ of types $II$ or $III$. But we firstly show they are not void. Let $S(u, \Sigma_p)\subset SX$ be a fibre of the second type. It contains an $(\frac{n}{4}+1)$-dimensional linear subspace $\Pi\subset X$ as seen in the proof of Corollary \ref{position-positive-intersection}. Then the lines in $S(u, \Sigma_p)\subset SX$ joining $\Pi$ and $\Sigma_p-Q_p$ are lines of type $III$. Note that codim$(\Pi, S(u, \Sigma_p))\geq 2$ if $n\geq 2$. A generic line in $S(u, \Sigma_p)$ not meeting with $\Pi$ is a line of type $II$. 

Then it again follows from Corollary \ref{position-positive-intersection} that $Q_p\cap Q_q$ has positive dimension for any $p, q\in L$. Therefore $L$ lies in the fibre of the second form with respect to $\pi_p$ for any $p\in L$. 
\end{proof}
%Use notations in the proof of Corollary \ref{position-positive-intersection}. The lines of type $II$ and $III$ are contained in some linear subspace $S(u, \Sigma_p)\subset SX$. Let $q$ be any point in $S(u, \Sigma_p)-(S(z, \Lambda)\cup \Sigma_p)$. The image of $S(z, \Lambda)$ under the projection in $S(u, \Sigma_p)$ along $q$ is the linear subspace $S(\Lambda, \tilde{\Lambda})\subset \Sigma_p$. Then we can see that a line $L$ passing through $q$ is of type $III$ if and only if the point where $L$ meets with $\Sigma_p$ is situated in $S(\Lambda, \tilde{\Lambda})$.

\begin{proposition}\label{lines-in-tangent} Let $X^n\subset \PP^N$ be a Severi variety of dimension $n>2$, and let $SX$ be the secant variety. The following statements hold.
\begin{itemize}
	\item $n>4$, any line $L\subset SX$ is contained in $\PT_x X$ for some $x\in X$.
	\item $n=4$, any line $L\subset SX$ except for the type II in Proposition \ref{5-types-line} is contained in $\PT_x X$ for some $x\in X$.
\end{itemize}
In particular, the point $x$ is unique for lines of type I in both cases.
\end{proposition}
\begin{proof}
The cases of types $IV$ and $V$ are easy to deal with. If $L\subset X$, of course $L\subset \PT_x X$ for any $x\in L$. If $L$ is a line secant to $X$, $L$ is also secant to some secant locus $Q$ . Suppose that $L$ intersects the quadric $Q$ at points $u, v$. Note that $Q$ is a nonsingular quadric of $\dim Q\geq 2$. There must exists a point $x\in Q$ such that $u, v\in \PT_x Q$. It follows that $L\subset \PT_x X$. If $u$ coincides with $v$, i.e. $L$ is tangent to $Q$, we just take $x=u$. 

Suppose that $L$ is of type $I$, which is contained in some fibre $S(u, \PT_x Q_p)$ of the first type. It follows that $L\subset \PT_x X$. By the result of Corollary \ref{position-positive-intersection} such point $x\in X$ is unique.

Suppose that $L$ is of type $II$. Then $L$ is contained in the linear subspace $S(q, \Sigma_p)$ for distinct $p, q\in L$. Let $\Lambda', \tilde{\Lambda}\subset Q_q$ be the $\frac{n}{4}$-dimensional subspaces introduced in the proof of Corollary \ref{position-positive-intersection}. For any $x\in \Lambda'\cap \tilde{\Lambda}$, we have 
\[
q\in S(\Lambda', \tilde{\Lambda})\subset \PT_x X.
\]
If $n\geq 8$, the dimension 
\[
\dim \Lambda'\cap \tilde{\Lambda}=\frac{n}{4}-1
\] 
is positive. Then the union $\bigcup_{x\in \Lambda'\cap \tilde{\Lambda}} \PT_x Q_p$ of tangent spaces covers the contact locus $\Sigma_p$. Hence there must exist some $x\in \Lambda'\cap \tilde{\Lambda}$ such that $p\in \PT_x X$. It follows that the line $L\subset \PT_x X$.

Suppose that $L$ is of type $III$. Then $L$ transversally intersects with $X$ at some point $z$. Then $L\subset S(z, \Sigma_p)$ for any $p\in L$. Let $\Lambda\subset Q_p$ denote the same notation in the proof of Corollary \ref{position-positive-intersection}. For any $x\in \Lambda$ the subspace $S(z, \Lambda)\subset \PT_x X$, and thus $z\in \PT_x X$. On the other hand, the dimension of $\Lambda$ is positive, the union $\bigcup_{x\in \Lambda} \PT_x Q_p$ of the tangent spaces covers the contact locus $\Sigma_p$. Hence there must exist some $x\in \Lambda$ such that $p\in \PT_x X$. It follows that the line $L$ is contained in $\PT_x X$.
\end{proof}

%For an irreducible cubic hypersurface $Y$ with dimension $d$, the expected dimension of the variety $F(Y)$ of line is $2d-4$. Such $F(Y)$ is a local complete intersection of the Grassmannian $Gr(2, d+2)$ of lines. This is always true for smooth cubic hyperusrface. If $Y$ has singularities, the dimension of $F(Y)$ may jump. For example,  $1)$ $Y$ is a cone over a smooth cubic curve; $2)$ $Y$ has codimension one singularities. 

Recall that for an irreducible cubic hypersurface $Y$ with dimension $d$, the expected dimension of the variety $F(Y)$ of line is $2d-4$. This is always true for smooth cubic hyperusrface. The dimension of $F(Y)$ may jump if $X$ is singular. %For example,  $1)$ $Y$ is a cone over a smooth cubic curve; $2)$ $Y$ has codimension one singularities.
In the following main theorem, we characterizes the Fano variety of lines on secant cubics. In particular, it has the expected dimension for each case, though the singularity of the secant cubic is relatively large.

\begin{theorem}\label{main-result}
Let $X^n\subset \PP^N, n=\frac{2(N-2)}{3}$ be a Severi variety with $n>2$, and let $SX\subset \PP^N$ be the secant variety. Denote by $F(SX)\subset Gr(2, N+2)$ the variety of lines on the cubic hyperusrface $SX$. For $n>2$ we use $\mathcal{C}^{I},\ldots, \mathcal{C}^{V}$ to denote the strata of lines of type $I,\ldots, V$ respectively. Then we have
\begin{itemize}
	\item If $n>4$, the variety $F(SX)$ is irreducible of the expected dimension $2(N-1)-4=3n-2$.
	\item If $n=4$, the strata of the variety $F(SX)$ satisfies
	\[
    \overline{\mathcal{C}^I}=\mathcal{C}^{I}\cup \mathcal{C}^{III}\cup \mathcal{C}^{IV}\cup\mathcal{C}^{V}, ~\overline{\mathcal{C}^{II}}=\mathcal{C}^{II}\cup \mathcal{C}^{III}\cup \mathcal{C}^{IV}\cup\mathcal{C}^{V}.
	\] 
	As a result, the variety $F(SX)$ has the expected dimension $10$. 
\end{itemize}
\end{theorem}
\begin{proof}
Let $\PT X$ be the embedded tangent bundle
\[
\PT X:=\{(x, p)\in X\times \PP^N ~|~ p\in \PT_x X\}
\]
of $X$ with the natural projections $\pi_X: \PT X\ra X$ and $\psi: \PT X\ra TX$. Let $Gr(\PT X/X)$ denote the Grassmannian of lines relatively to $\PT X\ra X$. It parametrizes lines lying in embedded tangent spaces $\PT_x X, x\in X$. Recall that $TX=SX$ for a Severi variety $X$. Hence the projection $\psi$ induces a morphism
\[
\Psi: Gr(\PT X/X)\ra F(SX).
\] 

If $n>4$, by Proposition \ref{lines-in-tangent} any line on $SX$ is contained in $\PT_x X$ for some $x\in X$, and the point $x$ is unique for lines of type $I$. Therefore the morphism $\Psi$ is surjective and birational in this case. The relative Grassmannian $Gr(\PT X/X)$ of lines is irreducible of dimension
\[
n+2(n-1)=2(N-1)-4.
\] 

If $n=4$, it follows from Proposition \ref{lines-in-tangent} that the image of $\Psi$ is equal to 
\[
\mathcal{C}^{I}\cup \mathcal{C}^{III}\cup \mathcal{C}^{IV}\cup\mathcal{C}^{V}.
\]
This union is irreducible since $Gr(\PT X/X)$ is irreducible with dimension equal to $10$. In order to prove the stratum $\mathcal{C}^{I}$ is Zariski open dense in it, we show that $\dim \mathcal{C}^{I}$ is $10$. Let $\Sigma:=\mathcal{C}^{III}\cup \mathcal{C}^{IV}\cup \mathcal{C}^{V}$. It is known that the point $[\ell]$ is smooth in $F(SX)$ if the line $\ell\subset SX$ does not meet with the singular part of $SX$ (cf. \cite[Thm. 4.2]{AK77}). Recall $SX$ is a cubic $7$-fold. Then the subvariety $F(SX)-\Sigma=\mathcal{C}^{I}\cup \mathcal{C}^{II}$ is smooth with (local) dimension equal to $2\cdot 7-4=10$. As a result, we have $\dim \mathcal{C}^{I}=2\cdot 7-4=10$ and 
\[
\overline{\mathcal{C}^I}=\mathcal{C}^{I}\cup \mathcal{C}^{III}\cup \mathcal{C}^{IV}\cup\mathcal{C}^{V}. 
\]

A line of type $II, III, IV$ or $V$ lies in some fibre of the second form in Lemma \ref{fibre-projection}. Fix one such fibre $S(u, \Sigma_p)$, the set of lines of type $II$ forms an open subset of lines in $S(u, \Sigma_p)$. Hence there is
\[
\overline{\mathcal{C}^{II}}=\mathcal{C}^{II}\cup \mathcal{C}^{III}\cup \mathcal{C}^{IV}\cup\mathcal{C}^{V}.
\]
As a consequence, the variety $F(SX)$ of line is the union $\overline{\mathcal{C}^{I}}\cup \overline{\mathcal{C}^{II}}$ with the expected dimension $10$
\end{proof}

In the rest of the section, we apply of our result to the ``geometric'' Torelli theorem for secant cubic hypersurface. The ``geometric'' Torelli theorem for smooth cubic hypersurfaces, due to F. Charles, is presented in the following.

\begin{theorem}\cite[Prop. 4]{charles12}\label{Charles-geo-Torelli}
Let $k$ be a field of characteristic different from $3$. Let $X$ and $X'$ be two cubic hypersurfaces of dimension $d\geq 3$ over $k$ with isolated singularities. Let $F$ and $F'$ be the Fano variety of lines on $X$ and $X'$ respectively. Let $g: F\ra F'$ be a projective isomorphism with respect to the Pl\"ucker embedding of $F$ and $F'$. Then there exists a projective isomorphism $f: X\ra X'$ inducing $g$.
\end{theorem}
The classical Torelli problem expects that projective manifolds are determined by the Hodge structures. The above theorem simply says the Fano variety of lines on a smooth cubic hypersurface determines the cubicis. Hence the Fano variety of lines plays the similar role of the Hodge structures. The Hodge structure defines for compact K\"ahler manifolds so that the classical Torelli theorem excludes singular varieties. But the variety of lines is always valid, which enables us to study the ``geometric'' Torelli theorem for singular cubics.

\begin{proposition}\label{Torelli-secant-cubic}
Let $X^n\subset \PP^N, n=\frac{2(N-2)}{3}$ be a Severi variety over a field $k$ of characteristic zero, and let $SX\subset \PP^N$ be the cubic secant hypersurface. Let $Y$ be a cubic hypersurface of dimension $N-1$. Denote by $F$ and $F'$ the Fano variety of lines on $SX$ and $Y$ respectively. Suppose there exists an isomorphism $g: (F,\OO_{F}(1))\ra (F', \OO_{F'}(1))$ with respect to the Pl\"ucker polarizations. Then there exists a projective isomorphism $f: SX\ra Y$ inducing $g$.
\end{proposition}

\begin{proof}
The proof basically repeats the proof of Theorem \ref{Charles-geo-Torelli}. The main input from Theorem \ref{main-result} is that $F(SX)$ has the expected dimension. Therefore $F(SX)$ is a local complete intersection of the Grassmannian of lines, and we can apply the Koszul complex method in Charles' argument. 

A projective isomorphism $SX\ra Y$ automatically induces a polarized isomorphism between the Fano varieties of lines. 

Let $V$ be an $(N+1)$-dimensional vector space so that $\PP^N=\PP(V)$. The polarized isomorphism $g: (F,\OO_{F}(1))\ra (F', \OO_{F'}(1))$ yields an automorphism $g'$ on $\PP(\wedge^2 V)$ fitting into the diagram
\[
\begin{tikzcd}
F \ar[r, hook] \ar[d, "g"']& \PP(\wedge^2 V) \ar[d, "g'"'] \\
F' \ar[r, hook] & \PP(\wedge^2 V).
\end{tikzcd}
\]
Let $G$ be the Grassmannian of lines on $\PP(V)$. We intend to show the automorphism $g'$ restricts to an automorphism of $G$. Let $S$ be the tautological rank two vector bundle on the Grassmannian $G$ of lines. The variety $F$ is the zero set of a section of the vector bundle $E:=Sym^3 S^{\vee}$ on $G$. By Theorem \ref{main-result} the dimension of $F$ is of the expected dimension, thus the section defining $F$ is a regular section, see \cite[Thm. 3.3 (iii)]{AK77}. Associated to a regular section is the Koszul complex
\[
\cdots \ra \wedge^2 E\ra \wedge^1 E\ra \OO_{G}\ra \OO_F\ra 0.
\]
By twisting the Koszul complex by $\OO(2)$ we obtain the spectral sequence 
\[
E^{p,q}_1:=H^q(G, \wedge^{-p} E(2))\Rightarrow H^{p+q}(F, \OO_F(2)).
\]
It follows from \cite[Thm. 5.1]{AK77} that 
\[
H^q(G, \wedge^{-p} E(2))=0, \text{~unless~} q=p=0,
\]
which implies $H^0(G, \OO_G(2))\overset{\sim}{\ra} H^0(F, \OO_F(2))$. The same conclusion holds for $F'$. As a consequence, the quadrics on $\PP(\wedge^2 V)$ defining $G$ is exactly the quadrics that vanishes on $F$ (resp. $F'$). Hence the autromorphism $g'$ sends $G$ to itself. 

Any automorphism on the Grassmannian of lines is induced by an autormophism of the projective spaces, see \cite[Thm. 10.19]{JoeHarris95}. Hence there exists an automorphism $f'$ on $\PP^N$ inducing $g'|_{G}$. The automorphism $f'$ gives an isomorphism $\Phi : P\ra P'$ where $P$ (resp. $P'$) is the incidence variety of lines on $SX$ (resp. $Y$) so that for any pair $(x, [\ell])\in P, x\in \ell\subset SX$ the morphism $\Phi$ is given by 
\[
\Phi(x, [\ell])=(f'(x), g([\ell])).
\]
Notice that the cubics $SX$ and $Y$ are dominant by the incidence varieties $P$ and $P'$ respectively. Therefore the isomorphism $\Phi$ or $f'$ restricts to an isomorphism $f: SX\ra Y$ inducing $g$.
\end{proof}

\end{document}